\documentclass[final,12pt]{elsarticle}

\makeatletter
\def\ps@pprintTitle{%
 \let\@oddhead\@empty
 \let\@evenhead\@empty
 \def\@oddfoot{}%
 \let\@evenfoot\@oddfoot}
\makeatother

\usepackage{mathrsfs}
\usepackage{graphicx}
\usepackage{enumerate}
\usepackage{mathtools,amsmath,amssymb}
\usepackage[none]{hyphenat}

\DeclareRobustCommand
  \ccdots{\mathinner{\cdotp\mkern-4mu\cdotp\mkern-4mu\cdotp}}

\newtheorem{theorem}{Theorem}[section]
\newtheorem{corollary}[theorem]{Corollary}

\newtheorem{proposition}[theorem]{Proposition}
\newtheorem{definition}[theorem]{Definition}

\newenvironment{proof}{\emph{Proof.}}{\hfill $\Box$ \\}

\newcommand{\comment}[1]{}

\begin{document}

\emergencystretch 3em

\begin{frontmatter}

\title{Fractional hypergraph isomorphism and fractional invariants}

\author{Flavia Bonomo-Braberman}
\address{Universidad de Buenos Aires. Facultad de Ciencias Exactas y Naturales. Departamento de Computaci\'on. / CONICET-Universidad de Buenos Aires. Instituto de
Investigaci\'on en Ciencias de la Computaci\'on (ICC). Buenos
Aires, Argentina.} \ead{fbonomo@dc.uba.ar}

\author{Dora Tilli}
\address{Universidad de Buenos Aires. Ciclo B\'asico Com\'un. \'Area de Matem\'atica. / Facultad de Ingenier\'{\i}a. Departamento de Matem\'atica. Buenos Aires,
Argentina.} \ead{dora.tilli@gmail.com}

\begin{abstract}
Fractional graph isomorphism is the linear relaxation of an
integer programming formulation of graph isomorphism. It preserves
some invariants of graphs, like degree sequences and equitable
partitions, but it does not preserve others like connectivity,
clique and independence numbers, chromatic number, vertex and edge
cover numbers, matching number, domination and total domination
numbers.

In this work, we extend the concept of fractional graph
isomorphism to hypergraphs, and give an alternative
characterization, analogous to one of those that are known for
graphs. With this new concept we prove that the fractional
packing,  covering, matching and transversal numbers on
hypergraphs are invariant under fractional hypergraph isomorphism.
As a consequence, fractional matching, vertex and edge cover,
independence, domination and total domination numbers are
invariant under fractional graph isomorphism. This is not the case
of fractional chromatic, clique, and clique cover numbers. In this
way, most of the classical fractional parameters are classified
with respect to their invariance under fractional graph
isomorphism.
\end{abstract}


\end{frontmatter}

\section{Introduction}\label{sec:introduction}

Graphs $G$ and $H$ are isomorphic, denoted $G \cong H$, when there
is a bijection $\phi : V(G) \to V(H)$ so that $uv \in E(G)$ if and
only if $\phi(u)\phi(v) \in E(H)$. In other words, graphs $G$ and
$H$ are isomorphic if they differ only in the names of their
vertices. In terms of matrices, if $A$ and $B$ are the adjacency
matrices of $G$ and $H$, then $G \cong H$ if and only if there is
a \emph{permutation matrix} $P$ so that $A = PBP^{-1}$. The
relation $A = PBP^{-1}$ can be rewritten as $AP = PB$, and the
requirement that $P$ is a permutation matrix can be restated as
``$P \cdot \textbf{1} = \textbf{1}$, $P^t \cdot \textbf{1} =
\textbf{1}$, and the entries of $P$ are in $\{0, 1\}$'', where
$P^t$ denotes the transposed matrix of $P$ and $\textbf{1}$ stands
for a vector of all $1$'s. So, the graph isomorphism problem can
be viewed as an integer programming feasibility problem, where $A$
and $B$ are given and the unknowns are the coefficients of matrix
$P$.

In~\cite{R-S-U-FGI}, Ramana, Scheinerman, and Ullman consider a
linear relaxation of the integer programming formulation and
denote the concept by \emph{fractional isomorphism} of two graphs.
Namely, they drop the requirement that $P$ is a $\{0,1\}$-matrix
and simply require the entries in $P$ to be nonnegative. A matrix
$S$ whose entries are nonnegative, and whose rows and columns all
sum up to 1 (i.e., $S \cdot \textbf{1} = \textbf{1}$ and $S^t
\cdot \textbf{1} = \textbf{1}$) is called a \emph{doubly
stochastic matrix}. Graphs $G$ and $H$ are said to be
\emph{fractionally isomorphic}, $G \cong_f H$, provided there is a
doubly stochastic matrix $S$ for which $AS = SB$ where $A$ and $B$
are the adjacency matrices of the graphs $G$ and $H$,
respectively.

The concept of fractional isomorphism fits within the more general
concept of fractional graph theory, surveyed by Ullman and
Scheinerman in~\cite{U-S-fgt}, in which fractional relaxations of
classical (integer) combinatorial optimization problems are
studied. Some of the notions involved in the main results about
fractional isomorphism are already present in the work of
Brualdi~\cite{Bru-doubsto}, Godsil~\cite{Gods-fgi},
Leighton~\cite{Lei-fgi}, McKay~\cite{McKay-PGI},
Mowshowitz~\cite{Mow-fgi}, and Tinhofer~\cite{Tin-fgi}, under
different names, mainly with the aim of having tools to
efficiently reject some instances of the isomorphism problem,
whose computational complexity is still open. The current best
result is a quasipolynomial time algorithm by
Babai~\cite{Babai-iso}. The color refinement procedure, introduced
in 1968 by Weisfeiler and Lehman~\cite{W-L-isocolor}, was also
related recently to fractional isomorphism
in~\cite{A-K-R-V-col-iso}.

It has been proved that the fractional isomorphism is an
equivalence relation that generalizes isomorphism, and preserves
some invariants of graphs as the number of vertices and edges, and
the degree sequence. Indeed, any two $r$-regular graphs on $n$
vertices are fractionally isomorphic. In particular, the disjoint
union of two triangles and a cycle of length six are fractionally
isomorphic. So, properties like connectivity, clique and
independence numbers, chromatic number, vertex, edge, and clique
cover numbers, and matching number, are not preserved. There are
also similar examples to show that neither are domination and
total domination numbers.

In this work, we extend the concept of fractional graph
isomorphism to hypergraphs, and give an alternative
characterization, analogous to one of those that are known for
graphs. With this new concept we prove that the fractional
packing,  covering, matching and transversal numbers on
hypergraphs are invariant under fractional hypergraph isomorphism.
As a consequence, fractional matching, vertex and edge cover,
independence, domination and total domination numbers are
invariant under fractional graph isomorphism. This is not the case
of fractional chromatic, clique, and clique cover numbers. In this
way, most of the classical fractional parameters are classified
with respect to their invariance under fractional graph
isomorphism.

\section{Definitions and basic results}\label{sec:defs}

\subsection{Graph theory}

All graphs in this work are finite, undirected, and have no loops
or multiple edges. For all graph-theoretic notions and notation
not defined here, we refer to West~\cite{West}. Let $G$ be a
graph. Denote by $V(G)$ its vertex set, and by $E(G)$ its edge
set. In general, we will assume $n = |V(G)|$ and $m = |E(G)|$.

Denote by $N(v)$ the neighborhood of a vertex $v$ in $G$, and by
$N[v]$ the closed neighborhood $N(v)\cup\{v\}$. If $X \subseteq
V(G)$, denote by $N(X)$ the set of vertices not in $X$ having at
least one neighbor in $X$. A vertex $v$ of $G$ is \emph{universal}
(resp. \emph{isolated}) if $N[v] = V(G)$ (resp. $N(v) =
\emptyset$).

Denote by $d(v)$ the degree of a vertex $v$ of $G$, i.e., $d(v)=
|N(v)|$, and by $\Delta(G)$ (resp. $\delta(G)$) the maximum (resp.
minimum) degree of a vertex in $G$. Let $S$ be a subset of the
vertex set of a graph $G$. Let $d(v,S)$ denote the degree of $v$
in $S$, i.e., $d(v, S) = |N(v) \cap S|$.

The \emph{degree sequence of a graph} $G$ is the multiset of the
degrees of its vertices. A graph is \emph{$k$-regular} if every
vertex has degree $k$, and \emph{regular} if it is $k$-regular for
some $k$.

The \emph{adjacency matrix} of a graph $G$ with vertices $v_1,
\dots, v_n$ is a matrix $A_G \in \{0,1\}^{n\times n}$ such that
$A_G(i,j) = 1$ if $v_iv_j \in E(G)$ and $A_G(i,j) = 0$, otherwise.
If the edges of $G$ are numbered as $e_1, \dots, e_m$, $m \geq 1$,
the \emph{vertex-edge incidence matrix} of $G$ is a matrix $M_G
\in \{0,1\}^{n\times m}$ such that $M_G(i,j)=1$ if $v_i$ is one of
the endpoints of $e_j$, and $M_G(i,j)=0$, otherwise. We will
denote by $J_{n}$ (resp. $J_{n\times m}$) the $(n\times n)$-matrix
(resp. $(n\times m)$-matrix) with all its entries equal $1$.

Given a graph $G$ and $W\subseteq V(G)$, denote by $G[W]$ the
subgraph of $G$ induced by $W$. Denote the size of a set $S$ by
$|S|$.

Denote by $C_n$ a chordless cycle on $n$ vertices.

A \emph{clique} or \emph{complete set} (resp.\ \emph{stable set}
or \emph{independent set}) is a set of pairwise adjacent (resp.\
non-adjacent) vertices. The size of a maximum size clique in a
graph $G$ is called the \emph{clique number} and denoted by
$\omega(G)$. The size of a maximum size independent set in a graph
$G$ is called the \emph{independence number} and denoted by
$\alpha(G)$.

A \emph{matching} of a graph is a set of pairwise disjoint edges
(i.e., no two edges share an endpoint). The maximum number of
edges of a matching in a graph $G$ is called the \emph{matching
number} and denoted by $\mu(G)$.

A \emph{vertex cover} is a set $S$ of vertices of a graph $G$ such
that each edge of $G$ has at least one endpoint in $S$.
Analogously, an \emph{edge cover} is a set $F$ of edges of a graph
$G$ such that each  vertex of $G$ belongs to at least one edge of
$F$. Denote by $\tau(G)$ (resp. $k(G)$) the size of a minimum
vertex (resp. edge) cover of a graph $G$, called the \emph{vertex
(edge) cover number} of $G$.

A \emph{clique cover} is a set $F$ of cliques of a graph $G$ such
that each vertex of $G$ belongs to at least one clique of $F$.
Denote by $\theta(G)$ the size of a minimum clique cover of a
graph $G$, called the \emph{clique cover number} of $G$.

A \emph{dominating set} in a graph $G$ is a set of vertices $S$
such that every vertex in $V(G)$ is either in $S$ or adjacent to a
vertex in $S$. The \emph{domination number} $\gamma(G)$ of a graph
$G$ is the size of a smallest dominating set. A \emph{total
dominating set} in a graph $G$ is a set of vertices $S$ such that
every vertex in $V(G)$ is adjacent to a vertex in $S$. The
\emph{total domination number} $\Gamma(G)$ of a graph $G$ is the
size of a smallest total dominating set.

A \emph{coloring} of a graph is an assignment of colors to its
vertices such that any two adjacent vertices are assigned
different colors. The smallest number $t$ such that $G$ admits
coloring with $t$ colors (a \emph{$t$-coloring}) is called the
\emph{chromatic number} of $G$ and is denoted by $\chi(G)$. A
coloring defines a partition of the vertices of the graph into
stable sets, called \emph{color classes}.

A graph is \emph{bipartite} if it admits a $2$-coloring, and a
\emph{bipartition} of it is a partition $(A,B)$ of its vertex set
into two stable sets. A bipartite graph $(A \cup B, E)$ is called
\emph{biregular} if the vertices of $A$ have the same degree $a$
and the vertices of $B$ have the same degree $b$ (where not
necessarily $a=b$). If we want to make explicit these values, we
write \emph{$(a,b)$-regular}. Given a graph $G$ and two disjoint
subsets $A$, $B$ of $V(G)$, the bipartite graph $G[A,B]$ is
defined as the subgraph of $G$ formed by the vertices $A \cup B$
and the edges of $G$ that have one endpoint in $A$ and one in $B$.
Notice that $G[A,B]$ is not necessarily an induced subgraph of
$G$.

\subsection{Matrix theory}

We will state here some well known definitions and results from
matrix theory which we need in this paper, and can be found, for
example, in~\cite{H-J-matrix}.

Every doubly stochastic matrix $S$ can be written as a convex
combination of permutation matrices, i.e., $S= \sum_{i \in I}
\alpha_i P_i$, where $\sum \alpha_i = 1$, the $\alpha_i$'s are
positive and each $P_i$ is a permutation matrix. This convex
combination is known as a \emph{Birkhoff decomposition} of $S$.

Let $A$ and $B$ be square matrices. The \emph{direct sum} of $A$
and $B$ is the square matrix

\begin{equation*}
A\oplus B=%
\begin{bmatrix}
A & 0 \\
1

0 & B%
\end{bmatrix}%
.
\end{equation*}

If $M=A\oplus B$ we say $M$ is \emph{decomposable}. In general,
$M$ is decomposable if there exists $A$, $B$, $P$ and $Q$ such
that $P$ and $Q$ are permutation matrices and $M=P(A\oplus B)Q$.
If no such decomposition exists, we say that $M$ is
\emph{indecomposable}.

Let $M$ be a $n\times n$ matrix. Let $D(M)$ be a digraph on $n$
vertices $v_{1},\dots,v_{n}$ with an arc from $v_{i}$ to $v_{j}$
if $M_{ij}\neq 0$. We say that $M$ is \emph{irreducible} when
$D(M)$ is strongly connected. Otherwise, we say that $M$ is
\emph{reducible}. Furthermore, we say that a matrix $M$ is
\emph{strongly irreducible} provided $PM$ is irreducible for any
permutation matrix $P$.

\begin{proposition}\label{lem:strong}\cite{U-S-fgt}
Let $S$ be a doubly stochastic, indecomposable matrix. Then $S$ is
also strongly irreducible.
\end{proposition}

\begin{theorem}\label{t1.8}\cite{H-L-P-convex}
Let $S$, $R$ be two doubly stochastic matrices of dimensions
$n\times n$ with Birkhoff's decomposition $S=\sum \alpha
_{i}P_{i}$ and $R=\sum \beta _{j}Q_{j}$, respectively. Let $x$,
$y$ be vectors of length $n$
\begin{enumerate}[(1)]
\item If $y=Sx$ and $x=Ry$, then $y=P_{i}x$ and $x=Q_{j}y$ for
every $i,j$.

\item Let $x$, $y$ as in (1). If, in addition, either $S$ or $R$
is indecomposable, then $x=y=s\cdot \bf{1}$ for some scalar $s$.

\item If $x$ and $y$ are $\{0,1\}$-vectors and $y=Sx$, then
$y=P_{i}x$ for every $i$.
\end{enumerate}
\end{theorem}

\subsection{Main results on fractional isomorphism}

We will survey here the main definitions and results on fractional
isomorphism, as they are stated in~\cite{R-S-U-FGI,U-S-fgt}. Some
of them were partially and independently shown
in~\cite{Bru-doubsto,Gods-fgi,Lei-fgi,McKay-PGI,Mow-fgi,Tin-fgi},
with different notations.

\begin{definition}
Let $G$, $H$ be graphs and $A_G$, $A_H$ their adjacency matrices,
respectively. We say that $G$ and $H$ are fractionally isomorphic,
and we write $G\cong_f H$, if there exists a doubly stochastic
matrix $S$ such that $A_G S=S A_H$.
\end{definition}

\begin{proposition}
The relation $\cong_f$ is an equivalence relation that preserves
the usual graph isomorphism.
\end{proposition}

\begin{proposition}
If $G \cong_f H$ for two graphs $G$ and $H$ then:
\begin{enumerate}
\item $G$ and $H$ have the same number of vertices;

\item $G$ and $H$ have the same number of edges;

\item $G$ and $H$ have the same degree sequence;

\item the adjacency matrices $A_G$ and $A_H$ have the same maximum
eigenvalue.
\end{enumerate}
\end{proposition}

\begin{proposition}
If $G$ and $H$ are two $r$-regular graphs with $n$ vertices then
$G\cong_f H$.
\end{proposition}

This implies that the disjoint union of two triangles and a cycle
of length six are fractionally isomorphic. However, $2C_3$ is not
connected, $\omega(2C_3) = \chi(2C_3) =  3$, $\alpha(2C_3) =
\theta(2C_3) = \mu(2C_3) = 2$, and $\tau(2C_3) = k(2C_3) = 4$,
while $C_6$ is connected, $\omega(C_6) = \chi(C_6) = 2$,
$\alpha(C_6) = \theta(C_6) = \mu(C_6) = 3$, $\tau(C_6) = k(C_6) =
3$. So the properties of connectivity, clique and independence
numbers, chromatic number, edge, vertex, and clique cover numbers,
and matching number, are not preserved by fractional isomorphism.

Similarly, $(C_5 \cup C_7) \cong_f C_{12}$. However, $\gamma(C_5
\cup C_7) = 5$, $\gamma(C_{12}) = 4$,
$\Gamma(C_5 \cup C_7) = 7$, and $\Gamma(C_{12}) = 6$. So, domination and total domination numbers are not preserved by fractional isomorphism. \\

The notion of fractional isomorphism is deeply related to the one
of equitable partition.
 We say that a partition $\{V_{1},  \dots, V_{s}\}$ of $V(G)$ is \emph{equitable} provided that for all $i$, $j$ and all $x$, $y$ $\in V_{i}$ we
have $d(x, V_{j}) = d(y, V_{j})$. In other words, each of the
induced subgraphs $G[V_{i}]$ must be regular and each of the
bipartite graphs $G[V_{i}, V_{j} ]$ must be biregular. It is clear
that every graph has an equitable partition: each vertex is a
class by itself. If $G$ is regular, then the singleton $\{V(G)\}$
is an equitable partition. Equitable partitions of a graph are
partially ordered by the usual \emph{refinement} relation for
partitions, i.e., if $P$ and $Q$ are partitions of a common set
$S$, we say that $P$ is a \emph{refinement} of $Q$ provided every
part of $P$ is a subset of some part in $Q$. When $P$ is a
refinement of $Q$ we also say that $P$ is \emph{finer} than $Q$
and that $Q$ is \emph{coarser} than $P$. The equitable partitions
of a graph form a lattice~\cite{McKay-PGI}. A maximum element of
the equitable partition lattice is denoted as a \emph{coarsest}
equitable partition of $G$.

\begin{theorem}~\cite{McKay-PGI}
Let $G$ be a graph. Then $G$ has a unique coarsest equitable
partition.
\end{theorem}

Let $G$ be a graph and let $P = \{P_{1},\dots,P_{p}\}$ be an
equitable partition of $V(G)$. The \emph{parameters} of $P$ are a
pair $(v,D)$ where $v$ is a $p$-vector whose $i$-th entry is the
size of $P_{i}$ and $D$ is a $(p\times p)$-matrix whose $ij$-entry
is $d(x,P_{j})$ for any $x \in P_{i}$. We say that equitable
partitions $P$ and $Q$ of graphs $G$ and $H$ have the same
parameters if we can index the sets in $P$ and $Q$ so that their
parameters $(v,D)$ are identical. In such a case we say that $G$
and $H$ have a \emph{common equitable partition}. If, in addition,
$P$ and $Q$ are coarsest equitable partitions of $G$ and $H$, then
we
say that $G$ and $H$ have a \emph{common coarsest equitable partition}.\\

Another concept which is central to the understanding of
fractional isomorphism is that of the \emph{iterated degree
sequence} of a graph. Let us recall that the \emph{degree} of a
vertex $v$ in $G$ is the cardinality of its neighbor set, $d(v) =
|N(v)|$, and the \emph{degree sequence of a graph} $G$ is the
multiset of the degrees of its vertices, $d_1(G) = \{d(v) : v \in
V(G)\}$. The \emph{degree sequence of a vertex} is the multiset of
the degrees of its neighbors: $d_1(v) = \{d(w) : w \in N(v)\}$. In
general, for $k \geq 1$, define: $d_{k+1}(G) = \{d_k(v) : v \in
V(G)\}$, and $d_{k+1}(v) = \{d_k(w) : w \in N(v)\}$. The
\emph{ultimate degree sequence} of a vertex $v$ or a graph $G$ are
defined as the infinite lists: $\mathscr{D}(v) = [ d_1(v), d_2(v),
\dots ]$ , and $\mathscr{D}(G) = [ d_1(G), d_2(G), \dots ]$. The
equivalence between having a fractional isomorphism, having a
common coarsest equitable partition, and having the same ultimate
degree sequence is the main theorem about fractional isomorphism.

\begin{theorem}\label{thm:frac-iso-main}
Let $G$ and $H$ be graphs. The following are equivalent.
\begin{enumerate}
\item $G \cong_f H$.

\item $G$ and $H$ have a common coarsest equitable partition.

\item $G$ and $H$ have some common equitable partition.

\item $\mathscr{D}(G) = \mathscr{D}(H)$.
\end{enumerate}
\end{theorem}

\section{Fractional hypergraph isomorphism}

An \emph{hypergraph} $G$ is a pair $G=(S,X)$ where $S = V(G)$ is a
finite set and $X = E(G)$ is a family of subsets of $S$. The set
$S$ is called the set of \emph{vertices} of the hypergraph. The
elements of $X$ are the \emph{hyperedges} or \emph{edges} of the
hypergraph. The \emph{degree} of a vertex is the number of
hyperedges which contains it. A hypergraph  is \emph{$r$-regular}
if every vertex has degree $r$, and it is \emph{$r$-uniform} if
all the hyperedges are of the same cardinality $r$. Graphs are the
$2$-uniform hypergraphs.

For a hypergraph $G$ with at least one edge, the
\emph{vertex-hyperedge incidence matrix} $M_G$ is the matrix
having $n$ rows (as vertices in the set $S$) and $m$ columns (as
hyperedges in the set $X$), and such that $M_{ij} = 1$ if the
vertex $i$ belongs to the hyperedge $j$ and $0$, otherwise. If the
hypergraph $G$ is a graph (i.e., 2-uniform), $M_G$ coincides with
the usual vertex-edge incidence matrix of the graph, so the
notation $M_G$ is well defined. The hypergraph $G$ is a graph if
and only if every column of $M_G$ has exactly two 1's. In other
words, $\textbf{1}^{t}\cdot M_G=2\cdot \textbf{1}^{t}$. On the
other hand, it is not difficult to see that $M_G\cdot
\textbf{1}=\delta$, where $\delta_{i}$ is the degree of the vertex
$i$ in $G$. The \emph{dual hypergraph} of a hypergraph $H$,
denoted $H^*$, is the hypergraph whose vertex-hyperedge incidence
matrix is $M_H^t$.
The \emph{2-section} of a hypergraph $H$ is the graph with the same vertex set as $H$, and such that two vertices are adjacent if and only if there is a hyperedge of $H$ that contains both of them.\\

In order to define a fractional hypergraph isomorphism notion, we
seek for a matrix equation characterizing the (hyper)graph
isomorphism  in terms of \emph{incidence} matrices instead of
adjacency matrices. Indeed, for two graphs $G$ and $H$, $G\cong H$
if and only if there exists permutation matrices $P_{1}$ and
$P_{2}$ such that $P_{1}M_{G}=M_{H}P_{2}^{t}$ and
$M_{G}P_{2}=P_{1}^{t}M_{H}$, and these equations hold also for
hypergraphs. We will linearly relax these conditions in the same
way as in the fractional graph isomorphism definition.

\begin{definition}
Let $G$ and $H$ be hypergraphs. We write $G\equiv H$ if either $G$
and $H$ have the same number of vertices and no hyperedges, or
their vertex-hyperedge incidence matrices $M_{G}$ are such that
there exist two doubly stochastic matrices $S_{1}$ and $S_{2}$ so
that $S_{1}M_{G}=M_{H}S_{2}^{t}$ and $M_{G}S_{2}=S_{1}^{t}M_{H}$.
\end{definition}

It is clear that if two matrices $M_{G}$ and $M_{H}$ follow the
conditions in the last definition, then they have the same
dimensions (both have $n$ rows and $m$ columns, for some $n, m$).

\begin{proposition}
The relation $\equiv $ is an equivalence relation.
\end{proposition}

\begin{proof}
It is straightforward for the case with no hyperedges. So we
assume there is at least one hyperedge. \emph{Reflexivity:}
$I_{n}M_{G}=M_{G}I_{m}^{t}$ and $M_{G}I_{m}=I_{n}^{t}M_{G}$, where
$I_{n}=I_{n}^{t}$ and $I_{m}=I_{m}^{t}$ are the identity matrices
of orders $n$ and $m$, respectively (and $n$ and $m$ are the
number of vertices and hyperedges of $G$, respectively).

\emph{Symmetry:} evident because renaming $S_{1}^{t}=U_{1}$ and
$S_{2}^{t}=U_{2}$ we obtain $U_{1}M_{H}=M_{G}U_{2}$ and
$M_{H}U_{2}=U_{1}^{t}M_{G}$.

\emph{Transitivity:} if $S_{1}M_{G}=M_{H}S_{2}^{t}$ and
$M_{G}S_{2}=S_{1}^{t}M_{H}$, and $S_{3}M_{G}=M_{K}S_{4}^{t}$ and
$M_{G}S_{4}=S_{3}^{t}M_{K}$, we have
$S_{1}M_{G}S_{4}=M_{H}S_{2}^{t}S_{4}$. Also
$S_{1}M_{G}S_{4}=S_{1}S_{3}^{t}M_{K}$, so we obtain
$M_{H}S_{2}^{t}S_{4}=S_{1}S_{3}^{t}M_{K}$. In the same way, we
have $S_{3}M_{G}S_{2}=S_{3}S_{1}^{t}M_{H}$ and also
$S_{3}M_{G}S_{2}=M_{K}S_{4}^{t}S_{2}$, so we obtain
$S_{3}S_{1}^{t}M_{H}=M_{K}S_{4}^{t}S_{2}$. Renaming
$U_{1}=S_{2}^{t}S_{4}$ and $U_{2}=S_{3}S_{1}^{t}$, we have
$M_{H}U_{1}=U_{2}^{t}M_{K}$ and $U_{2}M_{H}=M_{K}U_{1}^{t}$. (The
product of two doubly stochastic matrices is a doubly stochastic
matrix).
\end{proof}

\subsection{The relation $\equiv $ for 2-uniform hypergraphs}

We will show that the $\equiv$ relation for 2-uniform hypergraphs
is equivalent to the fractional isomorphism relation for graphs.
It is straightforward for the case of graphs with no edges, so,
from now on, we will assume that the graphs have at least one
edge.  First we will demonstrate it for regular graphs, then for
bipartite biregular graphs, and finally for general graphs.

\begin{proposition}
If $G$ and $H$ are $n$-vertex $r$-regular graphs. Then $G\equiv
H$.
\end{proposition}

\begin{proof}
As we mentioned above, we may assume $r>0$. We will compute
explicitly the matrices $S_{1}$ and $S_{2}$. Let us recall that
$J_{s}$ (resp. $J_{s\times t}$) is the $(s\times s)$-matrix (resp.
$(s\times t)$-matrix) with all its coefficients equal to 1. Let
$M_{G}$ and $M_{H}$ be the vertex-edge incidence matrices of the
graphs $G$ and $H$. The dimensions of these matrices are $n\times
a$, where $a=\frac{nr}{2}$ is the number of edges of $G$ and also
of $H$. We compute $S_{1}M_{G}$ and $M_{H}S_{2}^{t}$, where
$S_{1}=S_{1}^{t}=\frac{1}{n} J_{n}$ and
$S_{2}=S_{2}^{t}=\frac{1}{a}J_{a}$.

On the one hand,
$S_{1}M_{G}=\frac{1}{n}J_{n}M_{G}=\frac{1}{n}2J_{n\times a}$,
where the number $2$ appears because every column of $M_{G}$ adds
up to $2$ (it is a graph).

On the other hand,
$M_{H}S_{2}^{t}=M_{H}\frac{1}{a}J_{a}=\frac{1}{a}rJ_{n\times a}$,
where $r$ appears because every row of $M_{H}$ adds up to $r$
(there are exactly $r$ 1's in every row of an $r$-regular graph).

Since $a=\frac{nr}{2}$, we obtain $\frac{2}{n}=\frac{r}{a}$. So,
$S_{1}M_{G}=M_{H}S_{2}^{t}$. And, in the same way,
$M_{G}S_{2}=M_{G}\frac{1}{a}J_{a}=\frac{1}{a}rJ_{n\times a}$ and
$S_{1}^{t}M_{H}=\frac{1}{n}J_{n}M_{H}=\frac{1}{n}2J_{n\times a}$,
which coincide.
\end{proof}

\begin{proposition}
Let $G$ and $H$ be $(b,c)$-regular bipartite graphs, each of them
having $n$ vertices of degree $b$ and $m$ of degree $c$ (so, $
G\cong_f H$). Then $G\equiv H$.
\end{proposition}

\begin{proof}
Let $G$ and $H$ be $(b,c)$-regular bipartite graphs, each of them
having $n$ vertices of degree $b$ and $m$ of degree $c$. Then the
rows of $M_{G}$ and $M_{H}$ can be reordered so that, for each of
them, the first $n$ rows correspond to the vertices of degree $b$.
Let $a = nb = mc$ be the number of edges of $G$ and $H$, which we
assume greater than zero. After the reordering, $M_{G}$ can be
divided into two matrices of $n\times a$ and $m\times a$ (one
below the other) where in each of these matrices there is only one
1 per column, the first matrix has $b$ 1's per row, and the second
matrix has $c$ 1's per row. Namely,

$$M_{G}=%
\begin{bmatrix}
M_{G,1} \\
M_{G,2}%
\end{bmatrix}%
\mbox{ and } M_{H}=%
\begin{bmatrix}
M_{H,1} \\
M_{H,2}%
\end{bmatrix}%
.$$

We will compute explicitly the matrices $S_{1}$ and $S_{2}$. We
define

$$S_{1}=\frac{1}{n}J_{n}\oplus \frac{1}{m}J_{m}=%
\begin{bmatrix}
\frac{1}{n}J_{n} & 0 \\
0 & \frac{1}{m}J_{m}%
\end{bmatrix}%
=S_{1}^{t} \mbox{ and } S_{2}=\frac{1}{a}J_{a}=S_{2}^{t}.$$ Then,
$$S_{1}M_{G}=%
\begin{bmatrix}
\frac{1}{n}J_{n} & 0 \\
0 & \frac{1}{m}J_{m}%
\end{bmatrix}%
\begin{bmatrix}
M_{G,1} \\
M_{G,2}%
\end{bmatrix}%
=%
\begin{bmatrix}
\frac{1}{n}J_{n\times a} \\
\frac{1}{m}J_{m\times a}%
\end{bmatrix}%
\mbox{ and }$$

$$M_{H}S_{2}^{t}=%
\begin{bmatrix}
M_{H,1} \\
M_{H,2}%
\end{bmatrix}%
\frac{1}{a}J_{a}=\frac{1}{a}%
\begin{bmatrix}
bJ_{n\times a} \\
cJ_{m\times a}%
\end{bmatrix}%
$$
Since $\frac{b}{a}=\frac{1}{n}$ and $\frac{c}{a}=\frac{1}{m}$, the
right hand side expressions coincide, so
$S_{1}M_{G}=M_{H}S_{2}^{t}$. In the same way,
$$M_{G}S_{2}=%
\begin{bmatrix}
M_{G,1} \\
M_{G,2}%
\end{bmatrix}%
\frac{1}{a}J_{a}=\frac{1}{a}%
\begin{bmatrix}
bJ_{n\times a} \\
cJ_{m\times a}%
\end{bmatrix}%
\mbox{ and }$$ $$S_{1}^{t}M_{H}=%
\begin{bmatrix}
\frac{1}{n}J_{n} & 0 \\
0 & \frac{1}{m}J_{m}%
\end{bmatrix}%
\begin{bmatrix}
M_{H,1} \\
M_{H,2}%
\end{bmatrix}%
=%
\begin{bmatrix}
\frac{1}{n}J_{n\times a} \\
\frac{1}{m}J_{m\times a}%
\end{bmatrix}%
.$$ Again, the right hand side expressions coincide, so
$M_{G}S_{2}=S_{1}^{t}M_{H}$.
\end{proof}

\begin{proposition}
Let $G$ be a graph with the equitable coarsest partition consisting of two parts. Suppose $G$ can be described with the following parameters: $\ v=%
\begin{bmatrix}
v_1 \\
v_2%
\end{bmatrix}%
$ and $D=%
\begin{bmatrix}
D_{11} & D_{12} \\
D_{21} & D_{22}%
\end{bmatrix}%
$, with $v_1D_{12}=v_2D_{21}$. Let $H$ be another graph with the
same parameters ($ G\cong_f H$). Then $G\equiv H$.
\end{proposition}

\begin{proof}
Like for the previous propositions, we define explicitly doubly
stochastic matrices  $S_{1}$ and $S_{2}$. Let
\begin{equation*}
S_{1}=\frac{1}{v_1}J_{v_1}\oplus \frac{1}{v_2}J_{v_2}=%
\begin{bmatrix}
\frac{1}{v_1}J_{v_1} & 0 \\
0 & \frac{1}{v_2}J_{v_2}%
\end{bmatrix}%
=S_{1}^{t}
\end{equation*}
and
\begin{equation*}
S_{2}=\frac{1}{a_{1}}J_{a_{1}}\oplus \frac{1}{a_{2}}J_{a_{2}}\oplus \frac{1}{%
a_{12}}J_{a_{12}}=%
\begin{bmatrix}
\frac{1}{a_{1}}J_{a_{1}} & 0 & 0 \\
0 & \frac{1}{a_{2}}J_{a_{2}} & 0 \\
0 & 0 & \frac{1}{a_{12}}J_{a_{12}}%
\end{bmatrix}%
=S_{2}^{t}
\end{equation*}
where $a_{1}=\frac{v_1D_{11}}{2}$, $a_{2}=\frac{v_2D_{22}}{2}$,
and $a_{12}=v_1D_{12}=v_2D_{21}$ are the number of edges of $G$
and $H$ within the first part, the second part, and between both
parts, respectively. If some of them is zero, we do not include
the corresponding term in the direct sum. So, $M_{G}$ and $M_{H}$
have the same dimensions $n\times a$, where $n=v_1+v_2$ and
$a=a_{1}+a_{2}+a_{12}$. And, reordering them, we can write in
blocks:

$$M_{G}=%
\begin{bmatrix}
M_{G,11} & 0 & M_{G,12} \\
0 & M_{G,22} & M_{G,21}%
\end{bmatrix}%
$$ where $M_{G,11}$ (of dimension $v_1\times a_{1}$) corresponds
to the $D_{11}$-regular induced subgraph of $v_1$ vertices,
$M_{G,22}$ (of dimension $v_2\times a_{2}$) corresponds to the
$D_{22}$-regular induced subgraph of $v_2$ vertices, and
$M_{G,12}$ and $M_{G,21}$ (of dimension $v_1\times a_{12}$ and
$v_2\times a_{12}$, respectively) correspond to the bipartite
$(D_{12},D_{21})$-regular subgraph joining them. Notice that
$M_{G,11}\cdot \textbf{1}=D_{11} \cdot \textbf{1}$,
$\textbf{1}^t\cdot M_{G,11}=2\cdot \textbf{1}^t$, $M_{G,22}\cdot
\textbf{1}=D_{22}\cdot \textbf{1}$ and $\textbf{1}^t\cdot
M_{G,22}=2\cdot \textbf{1}^t$. Also, $M_{G,12}\cdot
\textbf{1}=D_{12}\cdot \textbf{1}$, $\textbf{1}^t\cdot
M_{G,12}=\textbf{1}^t$, $M_{G,21}\cdot \textbf{1}=D_{21}\cdot
\textbf{1}$ and $\textbf{1}^t\cdot M_{G,21}=\textbf{1}^t$. Then,

\begin{align*}
S_{1}M_{G} & =%
\begin{bmatrix}
\frac{1}{v_1}J_{v_1} & 0 \\
0 & \frac{1}{v_2}J_{v_2}%
\end{bmatrix}%
\begin{bmatrix}
M_{G,11} & 0 & M_{G,12} \\
0 & M_{G,22} & M_{G,21}%
\end{bmatrix} \\
& =%
\begin{bmatrix}
\frac{2}{v_1}J_{v_1\times a_{1}} & 0 & \frac{1}{v_1}J_{v_1\times a_{12}} \\
0 & \frac{2}{v_2}J_{v_2\times a_{2}} & \frac{1}{v_2}J_{v_2\times a_{12}}%
\end{bmatrix}%
\mbox{ and }
\end{align*}

\begin{align*}
M_{H}S_{2}^{t} & =%
\begin{bmatrix}
M_{H,11} & 0 & M_{H,12} \\
0 & M_{H,22} & M_{H,21}%
\end{bmatrix}%
\begin{bmatrix}
\frac{1}{a_{1}}J_{a_{1}} & 0 & 0 \\
0 & \frac{1}{a_{2}}J_{a_{2}} & 0 \\
0 & 0 & \frac{1}{a_{12}}J_{a_{12}}%
\end{bmatrix}\\
& =%
\begin{bmatrix}
\frac{D_{11}}{a_{1}}J_{v_1\times a_{1}} & 0 &
\frac{D_{12}}{a_{12}}J_{v_1\times a_{12}}
\\
0 & \frac{D_{22}}{a_{2}}J_{v_2\times a_{2}} & \frac{D_{21}}{a_{12}}J_{v_2\times a_{12}}%
\end{bmatrix}%
\end{align*}

\noindent Since $\frac{2}{v_1}=\frac{D_{11}}{a_{1}}$,
$\frac{2}{v_2}=\frac{D_{22}}{a_{2}}$,
$\frac{D_{12}}{a_{12}}=\frac{1}{v_1}$, and
$\frac{D_{21}}{a_{12}}=\frac{1}{v_2}$, the expressions above
coincide. Also,

\begin{align*}
S_{1}^{t}M_{H} & =%
\begin{bmatrix}
\frac{1}{v_1}J_{v_1} & 0 \\
0 & \frac{1}{v_2}J_{v_2}%
\end{bmatrix}%
\begin{bmatrix}
M_{H,11} & 0 & M_{H,12} \\
0 & M_{H,22} & M_{H,21}%
\end{bmatrix}\\
& =
\begin{bmatrix}
\frac{2}{v_1}J_{v_1\times a_{1}} & 0 & \frac{1}{v_1}J_{v_1\times a_{12}} \\
0 & \frac{2}{v_2}J_{v_2\times a_{2}} & \frac{1}{v_2}J_{v_2\times a_{12}}%
\end{bmatrix}%
\mbox{ and }
\end{align*}

\begin{align*}
M_{G}S_{2} & =%
\begin{bmatrix}
M_{G,11} & 0 & M_{G,12} \\
0 & M_{G,22} & M_{G,21}%
\end{bmatrix}%
\begin{bmatrix}
\frac{1}{a_{1}}J_{a_{1}} & 0 & 0 \\
0 & \frac{1}{a_{2}}J_{a_{2}} & 0 \\
0 & 0 & \frac{1}{a_{12}}J_{a_{12}}%
\end{bmatrix}\\
& =%
\begin{bmatrix}
\frac{D_{11}}{a_{1}}J_{v_1\times a_{1}} & 0 &
\frac{D_{12}}{a_{12}}J_{v_1\times a_{12}}
\\
0 & \frac{D_{22}}{a_{2}}J_{v_2\times a_{2}} & \frac{D_{21}}{a_{12}}J_{v_2\times a_{12}}%
\end{bmatrix}%
\end{align*}

\noindent  are equal. It is not hard to check that the equations
hold still when some of $a_1$, $a_2$, $a_{12}$ are zero and the
corresponding terms are missing in the direct sum.
\end{proof}

\subsubsection{General Case}

We show in the following two theorems that two graphs $G$ and $H$
are fractionally isomorphic if and only if  $G\equiv H$.

\begin{theorem}\label{thm:congf-equiv}
Let $G$ and $H$ be two fractionally isomorphic graphs ($G\cong_f
H$). Then $G\equiv H$.
\end{theorem}

\begin{proof}
If $G\cong_f H$, both graphs have the same number of vertices (say
$n$) and edges (say $a$), and the same coarsest equitable
partition. Thus, they can be described with the same parameters
$v$ and $D$ where $v$ is a vector of $k$ numbers corresponding to
the sizes of the $k$ parts of the graph and $D$ is a $(k\times
k)$-matrix whose coefficients correspond to the degrees of
connection within each part and between parts. Let $t=k+{k \choose
2}=\frac{k(k+1)}{2}$. We can classify the edges of each graph into
at most $t$ types, according to the blocks in which they have
their endpoints. Let $a_i = \frac{v_iD_{ii}}{2}$ for $1 \leq i
\leq k$ be the number of internal edges of the $i$-th part, and
$a_{ij} = v_iD_{ij}$ for $1 \leq i < j \leq k$ be the number of
edges joining the $i$-th and the $j$-th parts. Reordering rows and
columns, we can divide $M_{G}$ (whose dimension is $n\times a$)
into at most $kt$ blocks. For the first (at most) $k$ groups of
columns, there is one non-zero block, namely $M_{G,ii}$ for $1\leq
i\leq k$ such that $D_{ii} > 0$. The block $M_{G,ii}$ has
dimension $v_i \times a_i$ and corresponds to the regular
component $G[V_i]$ with at least one edge. It holds $M_{G,ii}\cdot
\textbf{1}=D_{ii}\cdot \textbf{1}$ and $\textbf{1}^t\cdot
M_{G,ii}=2\cdot \textbf{1}^t$. In the remaining groups of columns,
there are two non-zero blocks per group, namely $M_{G,ij}$ and
$M_{G,ji}$, for $1\leq i < j\leq k$ such that $D_{ij} > 0$, which
correspond to the bipartite subgraph $G[V_i,V_j]$  with at least
one edge. The blocks $M_{G,ij}$  and $M_{G,ji}$ are of dimension
$v_i \times a_{ij}$ and $v_j \times a_{ij}$, respectively. In this
case, $M_{G,ij}\cdot \textbf{1}= D_{ij} \cdot \textbf{1}$,
$\textbf{1}^t\cdot M_{G,ij}=\textbf{1}^t$, $M_{G,ji}\cdot
\textbf{1}= D_{ji} \cdot \textbf{1}$, and $\textbf{1}^t\cdot
M_{G,ji}=\textbf{1}^t$. We define $S_{1}= \frac{1}{v_{1}}J_{v_{1}}
\oplus \dots \oplus \frac{1}{v_{k}}J_{v_{k}}$ and
$S_{2}=\bigoplus_{\{1 \leq i \leq k  \ : \  a_i > 0\}}
\frac{1}{a_{i}}J_{a_{i}}\oplus \bigoplus_{\{1 \leq i < j \leq k \
: \  a_{ij} > 0\}} \frac{1}{a_{ij}}J_{a_{ij}}$ (this last part
ordered lexicographically by $(i,j)$). Then  $S_{1}=S_{1}^{t}$,
$S_{2}=S_{2}^{t}$, and

\medskip

{\scriptsize \setlength{\arraycolsep}{.5\arraycolsep}
$$S_{1}M_{G}= \begin{bmatrix}
\frac{1}{v_{1}}J_{v_{1}} & 0 &  \ccdots & 0 & 0 \\
0 & \frac{1}{v_{2}}J_{v_{2}} &  \ccdots & 0 & 0 \\
\ccdots & \ccdots & \ccdots & \ccdots & \ccdots \\
0 & 0 &  \ccdots & \frac{1}{v_{k-1}}J_{v_{k-1}} & 0 \\
0 & 0 &  \ccdots & 0 & \frac{1}{v_{k}}J_{v_{k}}%
\end{bmatrix}
\begin{bmatrix}
M_{G,11} & 0  & \ccdots & 0 & M_{G,12} & M_{G,13} & \ccdots & 0 \\
0 & M_{G,22}  & \ccdots & 0 & M_{G,21} & 0 & \ccdots & 0 \\
0 & 0 &  \ccdots & 0 & 0 & M_{G,31} & \ccdots & 0 \\
\ccdots &  \ccdots & \ccdots & \ccdots & \ccdots & \ccdots & \ccdots & \ccdots \\
0 & 0 & \ccdots & \ccdots & 0 & 0 & \ccdots & M_{G,(k-1)k} \\
0 & 0 & \ccdots & M_{G,kk} & 0 & 0 & \ccdots & M_{G,k(k-1)}%
\end{bmatrix}%
$$

$$\noindent
=\begin{bmatrix}
\frac{2}{v_{1}}J_{v_{1}\times a_{1}} & 0 & \ccdots & 0 & 0 & \frac{1}{v_{1}}%
J_{v_{1}\times a_{12}} & \frac{1}{v_{1}}J_{v_{1}\times a_{13}} &
\ccdots & 0
\\
0 & \frac{2}{v_{2}}J_{v_{2}\times a_{2}} & \ccdots & 0 & 0 & \frac{1}{v_{2}}%
J_{v_{2}\times a_{12}} & 0 & \ccdots & 0 \\
0 & 0 &  \ccdots & 0 & 0 & 0 & \frac{1}{%
v_{3}}J_{v_{3}\times a_{13}} & \ccdots & 0 \\
\ccdots & \ccdots & \ccdots & \ccdots & \ccdots & \ccdots & \ccdots & \ccdots & \ccdots \\
0 & 0 & \ccdots & \frac{2}{v_{k-1}}J_{v_{k-1}\times a_{k-1}} & 0 &
0 & 0 &
\ccdots & \frac{1}{v_{k-1}}J_{v_{k-1}\times a_{(k-1)k}} \\
0 & 0 & \ccdots & 0 & \frac{2}{v_{k}}J_{v_{k}\times a_{k}} & 0 & 0
& \ccdots &
\frac{1}{v_{k}}J_{v_{k}\times a_{(k-1)k}}%
\end{bmatrix}%
$$

\bigskip

$$M_{H}S_{2}^{t}=%
\begin{bmatrix}
M_{H,11} & 0 & \ccdots  & 0 & M_{H,12} & M_{H,13} & \ccdots & 0 \\
0 & M_{H,22} & \ccdots  & 0 & M_{H,21} & 0 & \ccdots & 0 \\
0 & 0 & \ccdots  & 0 & 0 & M_{H,31} & \ccdots & 0 \\
\ccdots & \ccdots &  \ccdots & \ccdots & \ccdots & \ccdots & \ccdots & \ccdots \\
0 & 0 &  \ccdots  & \ccdots & 0 & 0 & \ccdots & M_{H,(k-1)k} \\
0 & 0 &  \ccdots & M_{H,kk} & 0 & 0 & \ccdots & M_{H,k(k-1)}%
\end{bmatrix}%
\begin{bmatrix}
\frac{1}{a_{1}}J_{a_{1}} & 0 &  \ccdots & 0 \\
0 & \frac{1}{a_{2}}J_{a_{2}} &  \ccdots & 0 \\
0 & 0 &  \ccdots & 0 \\
\ccdots & \ccdots  & \ccdots & \ccdots  \\
0 & 0 &  \ccdots  & 0 \\
0 & 0 &  \ccdots & \frac{1}{a_{(k-1)k}}J_{a_{(k-1)k}}%
\end{bmatrix}%
$$

$$=
\begin{bmatrix}
\frac{D_{11}}{a_{1}}J_{v_{1}\times a_{1}} & 0 &  \ccdots  & 0 & \frac{%
D_{12}}{a_{12}}J_{v_{1}\times a_{12}} & \frac{D_{13}}{a_{13}}%
J_{v_{1}\times a_{13}} & \ccdots & 0 \\
0 & \frac{D_{22}}{a_{2}}J_{v_{2}\times a_{2}} &  \ccdots  & 0 & \frac{%
D_{21}}{a_{12}}J_{v_{2}\times a_{12}} & 0 & \ccdots & 0 \\
0 & 0 &  \ccdots  & 0 & 0 & \frac{%
D_{31}}{a_{13}}J_{v_{3}\times a_{13}} & \ccdots & 0 \\
\ccdots & \ccdots & \ccdots  & \ccdots & \ccdots & \ccdots & \ccdots & \ccdots \\
0 & 0 &  \ccdots  & 0
& 0 & 0 & \ccdots & \frac{D_{(k-1)k}}{a_{(k-1)k}}J_{v_{k-1}\times a_{(k-1)k}} \\
0 & 0 &  \ccdots  & \frac{D_{kk}}{a_{k}}J_{v_{k}\times a_{k}} & 0
& 0 & \ccdots
& \frac{D_{k(k-1)}}{a_{(k-1)k}}J_{v_{k}\times a_{(k-1)k}}%
\end{bmatrix}%
$$ }
\bigskip

So, $S_{1}M_{G}=M_{H}S_{2}^{t}$ because
$\frac{2}{v_{i}}=\frac{D_{ii}}{a_{i}}$ for $1\leq i\leq k$, $a_i >
0$ and $\frac{1}{v_{i}}=\frac{D_{ij}}{a_{ij}}$,
$\frac{1}{v_{j}}=\frac{D_{ji}}{a_{ij}}$  for $1\leq i < j \leq k$,
$a_{ij} > 0$. In the same way, we have
$S_{1}^{t}M_{H}=M_{G}S_{2}$.
\end{proof}

We will now show that if two hypergraphs $G$ and $H$ are graphs
(2-uniform hypergraphs) and  $G\equiv H$, then $G\cong_f H$. To do
this, we use the same ideas as in the main theorem of fractional
isomorphism.

\begin{theorem}\label{thm:equiv-congf}
If $G\equiv H$ and $G$ and $H$ are graphs then $G\cong_f H$.
\end{theorem}

\begin{proof}
We know that $S_{1}M_{G}=M_{H}S_{2}^{t}$ and
$M_{G}S_{2}=S_{1}^{t}M_{H}$ for some double stochastic matrices
$S_1$ and $S_2$. Firstly, we will show that we can assume that
$S_{1}$ and $S_{2}$ have a structure of blocks: $S_{1}=A_{1}\oplus
...\oplus A_{k}$ and $S_{2}=B_{1}\oplus ...\oplus B_{a}$ where
every $A_{i}$, $B_{i}$ is indecomposable (so strongly irreducible,
by Proposition~\ref{lem:strong}). The reason is that if
$PS_{1}Q=A_{1}\oplus ...\oplus A_{k}$ with $P$ and $Q$ permutation
matrices and $RS_{2}T=B_{1}\oplus ...\oplus B_{a}$ with $R$ and
$T$ permutation matrices, we can write:

\begin{equation*}
PS_{1}M_{G}R^{t}=PM_{H}S_{2}^{t}R^{t}
\end{equation*}
then
\begin{equation*}
(PS_{1}Q)(Q^{t}M_{G}R^{t})=(PM_{H}T)(T^{t}S_{2}^{t}R^{t})
\end{equation*}
but we have
\begin{equation*}
Q^{t}M_{G}S_{2}T=Q^{t}S_{1}^{t}M_{H}T
\end{equation*}
and then
\begin{equation*}
(Q^{t}M_{G}R^{t})(RS_{2}T)=(Q^{t}S_{1}^{t}P^{t})(PM_{H}T).
\end{equation*}
So, when we interchange rows and columns of $S_{1}$ and $S_{2}$ to
obtain a structure of blocks, $M_{G}$ and $M_{H}$ also interchange
rows and columns following theses rules: $M_{G}^{^{\prime
}}=Q^{t}M_{G}R^{t}$ and $M_{H}^{^{\prime }}=PM_{H}T$. We denote
$M_{G}$ and $M_{H}$ to these new matrices $M_{G}^{^{\prime }}$ and
$M_{H}^{^{\prime }}$.

Using the conditions $S_{1}M_{G}=M_{H}S_{2}^{t}$ and
$M_{G}S_{2}=S_{1}^{t}M_{H}$, we obtain:\\

\medskip

\noindent {\footnotesize
$$%
\begin{bmatrix}
A_{1} & 0 & ... & 0 \\
0 & A_{2} & ... & 0 \\
... & ... & ... & ... \\
0 & 0 & ... & A_{k}%
\end{bmatrix}%
\begin{bmatrix}
G_{11} & G_{12} & ... & G_{1a} \\
G_{21} & G_{22} & ... & G_{2a} \\
... & ... & ... & ... \\
G_{k1} & G_{k2} & ... & G_{ka}%
\end{bmatrix}%
=%
\begin{bmatrix}
H_{11} & H_{12} & ... & H_{1a} \\
H_{21} & H_{22} & ... & H_{2a} \\
... & ... & ... & ... \\
H_{k1} & H_{k2} & ... & H_{ka}%
\end{bmatrix}%
\begin{bmatrix}
B_{1}^{t} & 0 & ... & 0 \\
0 & B_{2}^{t} & ... & 0 \\
... & ... & ... & ... \\
0 & 0 & ... & B_{a}^{t}%
\end{bmatrix}%
$$

\medskip
\noindent
$$%
\begin{bmatrix}
G_{11} & G_{12} & ... & G_{1a} \\
G_{21} & G_{22} & ... & G_{2a} \\
... & ... & ... & ... \\
G_{k1} & G_{k2} & ... & G_{ka}%
\end{bmatrix}%
\begin{bmatrix}
B_{1} & 0 & ... & 0 \\
0 & B_{2} & ... & 0 \\
... & ... & ... & ... \\
0 & 0 & ... & B_{a}%
\end{bmatrix}%
=%
\begin{bmatrix}
A_{1}^{t} & 0 & ... & 0 \\
0 & A_{2}^{t} & ... & 0 \\
... & ... & ... & ... \\
0 & 0 & ... & A_{k}^{t}%
\end{bmatrix}%
\begin{bmatrix}
H_{11} & H_{12} & ... & H_{1a} \\
H_{21} & H_{22} & ... & H_{2a} \\
... & ... & ... & ... \\
H_{k1} & H_{k2} & ... & H_{ka}%
\end{bmatrix}%
$$
}
\medskip

Then $A_{i}G_{ij}=H_{ij}B_{j}^{t}$ and
$G_{ij}B_{j}=A_{i}^{t}H_{ij}$. Let us call $d_{ij}(G)=G_{ij}\cdot
\textbf{1}$ and $d_{ij}(H)=H_{ij}\cdot \textbf{1}$. We compute
$A_{i}G_{ij}\cdot \textbf{1}=H_{ij}B_{j}^{t}\cdot
\textbf{1}=H_{ij}\cdot \textbf{1}$.

Then $A_{i}d_{ij}(G)=d_{ij}(H)$ and as $G_{ij}\cdot
\textbf{1}=G_{ij}B_{j}\cdot \textbf{1}=A_{i}^{t}H_{ij}\cdot
\textbf{1}$ we have $d_{ij}(G)=A_{i}^{t}d_{ij}(H)$. But using
Theorem~\ref{t1.8}, we have that $d_{ij}(G)=d_{ij}(H)=c\cdot
\textbf{1}$ for some scalar $c$. On the other hand, renaming
$u_{ij}(G)=\textbf{1}^{t}\cdot G_{ij}$ and
$u_{ij}(H)=\textbf{1}^{t}\cdot H_{ij}$, we obtain that $u_{ij}(G)$
$=$ $\textbf{1}^{t}\cdot G_{ij}$ $=$ $\textbf{1}^{t}\cdot
A_{ij}G_{ij}$ $=$ $\textbf{1}^{t}\cdot H_{ij}B_{j}^{t}$ $=$
$u_{ij}(H)B_{ij}^{t}$ and $u_{ij}(G)B_{j}$ $=$
$\textbf{1}^{t}\cdot G_{ij}B_{j}$ $=$ $\textbf{1}^{t}\cdot
A_{i}^{t}H_{ij}$ $=$ $\textbf{1}^{t}\cdot H_{ij}$ $=$ $u_{ij}(H)$.
And again using Theorem~\ref{t1.8} for $u_{ij}(G)^t$ and
$u_{ij}(H)^t$, it holds that $u_{ij}(G)=u_{ij}(H)=e\cdot
\textbf{1}^{t}$. For 2-uniform hypergraphs, this last value $e$
could only be $0$, $1$ or $2$. The matrices $M_{G}$ and $M_{H}$
are divided into blocks $G_{ij}$ and $H_{ij}$ of the same
dimension where every row adds up the same and every column adds
up the same (that only could be $2$, $1$ or $0$ in the case of
$2$-uniform hypergraphs). Since the columns of $M_{G}$ and $M_{H}$
add up to $2$, we have the following cases: if the sum of the
columns of $G_{ij}$ (and also of $H_{ij}$) adds up to $2$, we
have the vertices of the regular classes, if the sum of the
columns of $G_{ij}$ adds up to $0$ then $G_{ij}$ is the null
matrix (so also is $H_{ij}$) and if the sum of the columns of
$G_{ij}$ adds up to $1$, we have another $G_{kj}$ where the sum is
also $1$, we have the connections between class $i$ and class $k$
(and the same happens to $H_{ij}$). In conclusion, $S_{1}$ and
$S_{2}$ induce the same equitable partition of the vertices and
edges of $G$ and $H$, respectively, so by
Theorem~\ref{thm:frac-iso-main}, $G\cong_f H$.
\end{proof}

As we show that, for graphs, $G\cong_f H $ if and only if $G\equiv
H$, we will use the notation $G\cong_f H$ from now on, instead of
$G\equiv H$. The concept of fractional isomorphism can be extended
to hypergraphs, and we denote $G\cong_f H$ if $G\equiv H$.

Likewise, in the case of graphs, the conditions that define the
fractional hypergraph isomorphism can be verified using a linear
programming model of polynomial size in the size of the
hypergraphs: $G\cong_f H$ if and only if there exist $S_{1}$ and
$S_{2}$ doubly stochastic matrices such that
$S_{1}M_{G}=M_{H}S_{2}^{t}$ and $M_{G}S_{2}=S_{1}^{t}M_{H}$. By
the polynomiality of linear programming~\cite{Kha-lp}, the
recognition of the fractional isomorphism of hypergraphs is
polynomial.

\subsubsection{Basic properties of fractional isomorphism of hypergraphs}

We present here the basic properties of fractional isomorphism of
hypergraphs, which are similar to those of graphs but also involve
hyperedge sizes, that are implicit for graphs since graphs are
$2$-uniform.

\begin{proposition}\label{basics}
If $G \cong_f H$ for two hypergraphs $G$ and $H$ then:
\begin{enumerate}
\item $G$ and $H$ have the same number of vertices;

\item $G$ and $H$ have the same number of hyperedges;

\item $G$ and $H$ have the same degree sequence;

\item $G$ and $H$ have the same multiset of hyperedge sizes;

\item $G^* \cong_f H^*$ (their dual hypergraphs are fractionally
isomorphic).
\end{enumerate}
\end{proposition}

\begin{proof}
If $G\cong_f H$ then there exist $S_{1}$ and $S_{2}$ doubly
stochastic matrices such that $S_{1}M_{G}=M_{H}S_{2}^{t}$ and
$M_{G}S_{2}=S_{1}^{t}M_{H}$. In particular, to make the matrix
products well defined, $M_G$ and $M_H$ have to have the same
dimensions, so $G$ and $H$ have the same number of vertices and
hyperedges.

The degree sequence $d_G$ of $G$ can be obtained by multiplying
$M_G \cdot \textbf{1}$, while the multiset of hyperedge sizes
$u_G^t$ can be obtained by multiplying $\textbf{1}^t \cdot M_G$.

\begin{eqnarray*}
M_{G}S_{2}\cdot \textbf{1} & = & S_{1}^{t}M_{H} \cdot \textbf{1}\\
M_{G}\cdot \textbf{1} & = & S_{1}^{t} \cdot d_{H}\\
d_{G}& = & S_{1}^{t} \cdot d_{H}\\\end{eqnarray*}

Similarly, $S_1 \cdot d_G = d_H$. By Theorem~\ref{t1.8}, $d_{G}$
is a permutation of $d_{H}$.

Analogously,

\begin{eqnarray*}
\textbf{1}^t \cdot S_{1} M_{G} & = & \textbf{1}^t \cdot M_{H} S_{2}^{t}\\
\textbf{1}^t \cdot M_{G} & = & u_{H}^t \cdot S_{2}^{t}\\
u_G^t & = & u_{H}^t \cdot S_{2}^{t}\\
u_G & = & S_{2} \cdot u_{H}
\end{eqnarray*}

Similarly, $S_2^t \cdot u_G = u_H$. By Theorem~\ref{t1.8}, $u_{G}$
is a permutation of $u_{H}$.

Finally, since $M_{G^*} = M_G^t$, $M_{H^*} = M_H^t$, and
transposing the equations we have $M_{G}^{t}S_1^{t}=S_2M_{H}^{t}$
and $S_2^{t}M_{G}^{t}=M_{H}^{t}S_{1}$, it follows that $G^*
\cong_f H^*$.
\end{proof}

As a corollary, we have the following.

\begin{corollary}\label{thm:frac-iso-hyper-graph}
Let $G$ and $H$ be hypergraphs. If $G \cong_f H$ and $G$ is a
graph, then $H$ is also a graph.
\end{corollary}

So, we can strengthen Theorem~\ref{thm:equiv-congf}, requiring
only that $G$ be a graph, because $H$ will necessarily be a graph.

We can also extend the concept of equitable partition to
hypergraphs.

Let $H$ be an hypergraph. Let $P = \{V_{1}, \dots, V_{s}, X_1,
\dots, X_r\}$ be a partition of $V(H)$ and $E(H)$ (each $V_i$ is a
subset of $V(H)$ and each $X_j$ of $E(H)$, $r$ can be zero if $H$
has no hyperedges). The partition $P$ is \emph{equitable} if, for
every $1 \leq i \leq s$ and every $1 \leq j \leq r$, every vertex
of $V_i$ belongs to the same number of hyperedges of $X_j$, and
every hyperedge of $X_j$ contains the same number of vertices of
$V_i$.

Every hypergraph has a trivial equitable partition: each vertex
and each hyperedge is a class by itself. If $H$ is uniform and
regular, then $\{V(H),E(H)\}$ is an equitable partition.

Notice that, if we have an equitable partition $P = \{V_{1},
\dots, V_{s},$ $X_1, \dots, X_r\}$ and, for some $1 \leq j, j'\leq
r$ and for every $1 \leq i \leq s$, the hyperedges of $X_j$ and
$X_{j'}$ have the same number $n_i$ of vertices of the class
$V_i$, then we can define a coarser equitable partition replacing
$X_j$ and $X_{j'}$ by $X_j \cup X_{j'}$.

In particular, if $G$ is a graph, for every equitable partition $P
= \{V_{1}, \dots, V_{s},$ $X_1, \dots, X_r\}$, we have a coarser
one with the same vertex partition and so that each edge set is
either the set of all edges having both endpoints in some vertex
set $V_i$, or the set of all edges having one endpoint in $V_i$
and the other in $V_j$, for some $i,j$. So, the concept of
equitable partition for hypergraphs coincides in graphs with the
traditional definition.

For hypergraphs, the \emph{parameters} of an equitable partition
$P = \{V_{1}, \dots, V_{s},$ $X_1, \dots, X_r\}$ are a triple
$(v,D,U)$ where $v$ is a $s$-vector whose $i$-th entry is the size
of $V_{i}$, $D$ and $U$ are $(s\times r)$-matrices such that
$D_{ij}$ is the number of hyperedges of $X_j$ to which a vertex of
$V_i$ belongs, and $U_{ij}$ is the number of vertices of $V_i$
that an hyperedge of $X_j$ contains. When $r > 0$, the number of
edges $a_j$ in each $X_j$, which is always greater than zero, is
then $\frac{v_i D_{ij}}{U_{ij}}$, for any $i$ such that $U_{ij} >
0$.

We say that equitable partitions $P$ and $Q$ of hypergraphs $G$
and $H$ have the same parameters if we can index the sets in $P$
and $Q$ so that their parameters $(n,D,U)$ are identical. In such
a case we say that $G$ and $H$ have a \emph{common equitable
partition}.

Following the main ideas in the proofs of
Theorems~\ref{thm:congf-equiv} and~\ref{thm:equiv-congf}, we can
prove the following result.

\begin{theorem}\label{thm:frac-iso-hyper-main}
Let $G$ and $H$ be hypergraphs. Then $G \cong_f H$ if and only if
$G$ and $H$ have a common equitable partition.
\end{theorem}

\begin{proof}
Suppose $G$ and $H$ have a common equitable partition with
parameters $(v,D,U)$. In particular, they have the same number of
vertices and hyperedges. Let $a_j$ be the number of edges of
$X_j$, for $1 \leq j \leq r$. Recall that for every $1 \leq i \leq
s$, $U_{ij}a_{j} = D_{ij}v_{i}$. Reordering rows and columns
according to the partition on each hypergraph, we can divide
$M_{G}$ into $sr$ blocks. For each block $M_{G,ij}$ of dimension
$v_i \times a_j$, with $1\leq i\leq s$, $1\leq j\leq r$,
$M_{G,ij}\cdot \textbf{1}=D_{ij}\cdot \textbf{1}$ and
$\textbf{1}^t\cdot M_{G,ij}=U_{ij}\cdot \textbf{1}^t$.

We define $S_{1}=\frac{1}{v_{1}} J_{v_{1}}\oplus \dots \oplus
\frac{1}{v_{s}}J_{v_{s}}$ ($S_{1}=S_{1}^{t}$) and  $S_{2}=$
$\frac{1}{a_{1}}J_{a_{1}}\oplus \dots \oplus
\frac{1}{a_{r}}J_{a_{r}}$ ($S_{2}=S_{2}^{t}$). Then

\medskip

\begin{align*}
S_{1}M_{G} & = \begin{bmatrix}
\frac{1}{v_{1}}J_{v_{1}} & 0 &  \ccdots & 0 \\
0 & \frac{1}{v_{2}}J_{v_{2}} &  \ccdots & 0 \\
\ccdots & \ccdots & \ccdots & \ccdots  \\
0 & 0 &  \ccdots & \frac{1}{v_{k}}J_{v_{k}}%
\end{bmatrix}
\begin{bmatrix}
M_{G,11} & M_{G,12} & \ccdots & M_{G,1r} \\
M_{G,21} & M_{G,22} & \ccdots & M_{G,2r} \\
\ccdots &  \ccdots & \ccdots & \ccdots  \\
M_{G,s1} & M_{G,s2} & \ccdots & M_{G,sr}%
\end{bmatrix}\\
& =\begin{bmatrix} \frac{U_{11}}{v_{1}}J_{v_{1}\times a_{1}} &
\frac{U_{12}}{v_{1}}J_{v_{1}\times a_{2}} & \ccdots  &
\frac{U_{1r}}{v_{1}}J_{v_{1}\times a_{r}}
\\
\frac{U_{21}}{v_{2}}J_{v_{2}\times a_{1}} &
\frac{U_{22}}{v_{2}}J_{v_{2}\times a_{2}} & \ccdots  &
\frac{U_{2r}}{v_{2}}J_{v_{2}\times a_{r}}
\\
\ccdots & \ccdots & \ccdots & \ccdots \\
\frac{U_{s1}}{v_{s}}J_{v_{s}\times a_{1}} &
\frac{U_{s2}}{v_{s}}J_{v_{s}\times a_{2}} & \ccdots  &
\frac{U_{sr}}{v_{s}}J_{v_{s}\times a_{r}}
\end{bmatrix}%
\end{align*}

\begin{align*}
M_{H}S_{2}^{t} & =%
\begin{bmatrix}
M_{H,11} & M_{H,12} & \ccdots  & M_{H,1r} \\
M_{H,21} & M_{H,22} & \ccdots  & M_{H,2r} \\
\ccdots & \ccdots & \ccdots & \ccdots  \\
M_{H,s1} & M_{H,s2} & \ccdots & M_{H,sr}%
\end{bmatrix}%
\begin{bmatrix}
\frac{1}{a_{1}}J_{a_{1}} & 0 &  \ccdots & 0 \\
0 & \frac{1}{a_{2}}J_{a_{2}} &  \ccdots & 0 \\
\ccdots & \ccdots  & \ccdots & \ccdots  \\
0 & 0 &  \ccdots & \frac{1}{a_{r}}J_{a_{r}}%
\end{bmatrix}\\
& =
\begin{bmatrix}
\frac{D_{11}}{a_{1}}J_{v_{1}\times a_{1}} & \frac{D_{12}}{a_{2}}J_{v_{1}\times a_{2}} &  \ccdots  & \frac{%
D_{1r}}{a_{r}}J_{v_{1}\times a_{r}} \\
\frac{D_{21}}{a_{1}}J_{v_{2}\times a_{1}} & \frac{D_{22}}{a_{2}}J_{v_{2}\times a_{2}} &  \ccdots  & \frac{%
D_{2r}}{a_{r}}J_{v_{2}\times a_{r}} \\
\ccdots & \ccdots & \ccdots  & \ccdots \\
\frac{D_{s1}}{a_{1}}J_{v_{s}\times a_{1}} & \frac{D_{s2}}{a_{2}}J_{v_{s}\times a_{2}} &  \ccdots  & \frac{%
D_{sr}}{a_{r}}J_{v_{s}\times a_{r}}%
\end{bmatrix}%
\end{align*}

So, $S_{1}M_{G}=M_{H}S_{2}^{t}$ because
$\frac{U_{ij}}{v_{i}}=\frac{D_{ij}}{a_{j}}$ for $1\leq i\leq s$
and $1 \leq j \leq r$. In the same way, we have
$S_{1}^{t}M_{H}=M_{G}S_{2}$.

Now, suppose $G \cong_f H$. We know that
$S_{1}M_{G}=M_{H}S_{2}^{t}$ and $M_{G}S_{2}=S_{1}^{t}M_{H}$ for
some double stochastic matrices $S_1$ and $S_2$. Repeating the
reasoning in the proof of Theorem~\ref{thm:equiv-congf}, we can
assume that $S_{1}$ and $S_{2}$ have a structure of blocks:
$S_{1}=A_{1}\oplus ...\oplus A_{s}$ and $S_{2}=B_{1}\oplus
...\oplus B_{r}$ where every $A_{i}$, $B_{i}$ is strongly
irreducible. Let $v_i$ be the number of rows and columns of $A_i$,
for $1 \leq i \leq s$, and $a_i$ be the number of rows and columns
of $B_i$, for $1 \leq i \leq r$. We can partition $M_G$ and $M_H$
into ($v_i \times a_j$)-blocks $G_{ij}$ and $H_{ij}$,
respectively. Using the conditions $S_{1}M_{G}=M_{H}S_{2}^{t}$ and
$M_{G}S_{2}=S_{1}^{t}M_{H}$, we obtain:\\

\medskip

\noindent {\footnotesize
$$%
\begin{bmatrix}
A_{1} & 0 & ... & 0 \\
0 & A_{2} & ... & 0 \\
... & ... & ... & ... \\
0 & 0 & ... & A_{s}%
\end{bmatrix}%
\begin{bmatrix}
G_{11} & G_{12} & ... & G_{1r} \\
G_{21} & G_{22} & ... & G_{2r} \\
... & ... & ... & ... \\
G_{s1} & G_{s2} & ... & G_{sr}%
\end{bmatrix}%
=%
\begin{bmatrix}
H_{11} & H_{12} & ... & H_{1r} \\
H_{21} & H_{22} & ... & H_{2r} \\
... & ... & ... & ... \\
H_{s1} & H_{s2} & ... & H_{sr}%
\end{bmatrix}%
\begin{bmatrix}
B_{1}^{t} & 0 & ... & 0 \\
0 & B_{2}^{t} & ... & 0 \\
... & ... & ... & ... \\
0 & 0 & ... & B_{r}^{t}%
\end{bmatrix}%
$$

\medskip
\noindent
$$%
\begin{bmatrix}
G_{11} & G_{12} & ... & G_{1r} \\
G_{21} & G_{22} & ... & G_{2r} \\
... & ... & ... & ... \\
G_{s1} & G_{s2} & ... & G_{sr}%
\end{bmatrix}%
\begin{bmatrix}
B_{1} & 0 & ... & 0 \\
0 & B_{2} & ... & 0 \\
... & ... & ... & ... \\
0 & 0 & ... & B_{r}%
\end{bmatrix}%
=%
\begin{bmatrix}
A_{1}^{t} & 0 & ... & 0 \\
0 & A_{2}^{t} & ... & 0 \\
... & ... & ... & ... \\
0 & 0 & ... & A_{s}^{t}%
\end{bmatrix}%
\begin{bmatrix}
H_{11} & H_{12} & ... & H_{1r} \\
H_{21} & H_{22} & ... & H_{2r} \\
... & ... & ... & ... \\
H_{s1} & H_{s2} & ... & H_{sr}%
\end{bmatrix}%
$$
}
\medskip

Then $A_{i}G_{ij}=H_{ij}B_{j}^{t}$ and
$G_{ij}B_{j}=A_{i}^{t}H_{ij}$. Let us call $d_{ij}(G)=G_{ij}\cdot
\textbf{1}$ and $d_{ij}(H)=H_{ij}\cdot \textbf{1}$. We compute
$A_{i}G_{ij}\cdot \textbf{1}=H_{ij}B_{j}^{t}\cdot
\textbf{1}=H_{ij}\cdot \textbf{1}$.

Then $A_{i}d_{ij}(G)=d_{ij}(H)$ and as $G_{ij}\cdot
\textbf{1}=G_{ij}B_{j}\cdot \textbf{1}=A_{i}^{t}H_{ij}\cdot
\textbf{1}$ we have $d_{ij}(G)=A_{i}^{t}d_{ij}(H)$. But using
Theorem~\ref{t1.8}, we have that $d_{ij}(G)=d_{ij}(H)=D_{ij}\cdot
\textbf{1}$ for some scalar $D_{ij}$. On the other hand, renaming
$u_{ij}(G)=\textbf{1}^{t}\cdot G_{ij}$ and
$u_{ij}(H)=\textbf{1}^{t}\cdot H_{ij}$, we obtain that $u_{ij}(G)$
$=$ $\textbf{1}^{t}G_{ij}$ $=$ $\textbf{1}^{t}A_{ij}G_{ij}$ $=$
$\textbf{1}^{t}H_{ij}B_{j}^{t}$ $=$ $u_{ij}(H)B_{ij}^{t}$ and
$u_{ij}(G)B_{j}$ $=$ $\textbf{1}^{t}G_{ij}B_{j}$ $=$
$\textbf{1}^{t}A_{i}^{t}H_{ij}$ $=$ $\textbf{1}^{t}H_{ij}$ $=$
$u_{ij}(H)$. And again using Theorem~\ref{t1.8} for $u_{ij}(G)^t$
and $u_{ij}(H)^t$, it holds that $u_{ij}(G)=u_{ij}(H)=U_{ij}\cdot
\textbf{1}^{t}$ for some scalar $U_{ij}$. Thus, the partition of
the vertices and of the hyperedges of $G$ and $H$ induced by the
blocks $G_{ij}$ and $H_{ij}$ of the respective incidence matrices
is a common equitable partition of $G$ and $H$.
\end{proof}

\begin{corollary}
If $G$ and $H$ are two $k$-uniform $r$-regular hypergraphs with
$n$ vertices, then $G\cong_f H$.
\end{corollary}

Notice that the fact that two hypergraphs are fractionally
isomorphic, does not imply that their 2-sections are fractionally
isomorphic graphs. Consider the following $4$-uniform $2$-regular
hypergraphs on $8$ vertices $v_1,\dots,v_8$ (thus, fractionally
isomorphic). The hypergraph $H$ has hyperedges
$\{v_1,v_2,v_3,v_4\}$, $\{v_3,v_4,v_5,v_6\}$,
$\{v_5,v_6,v_7,v_8\}$, $\{v_1,v_2,v_7,v_8\}$; the hypergraph $G$
has hyperedges $\{v_1,v_2,v_3,v_4\}$, $\{v_1,v_2,v_3,v_8\}$,
$\{v_4,v_5,v_6,v_7\}$, $\{v_5,v_6,v_7,v_8\}$. The $2$-section of
$H$ is $5$-regular, while the $2$-section of $G$ has $6$ vertices
of degree $4$ and $2$ vertices of degree $6$, so they are not
fractionally isomorphic.

\subsubsection{Bipartite representation and fractional isomorphism}

It is known that every hypergraph can be described by a bipartite
graph where the two parts of the bipartition correspond to the
vertices and hyperedges of the hypergraph, respectively.

A possible question arises with this correspondence: Do two
fractionally isomorphic hypergraphs correspond to fractionally
isomorphic bipartite graphs? And if we have two bipartite graphs
that arise from hypergraphs, if they are fractionally isomorphic,
are the hypergraphs fractionally isomorphic?

Given a hypergraph $G$, we may construct a bipartite graph $B_{G}$
where the first part of $B_{G}$ is in correspondence with the
vertices of $G$ and the second part of $B_{G}$ is in
correspondence with the hyperedges of $G$. We have an edge in
$B_{G}$ that links a vertex in the first part to a vertex in the
second part if and only if the corresponding vertex of $G$ belongs
to the corresponding hyperedge in $G$. If $M_{G}$ is the
vertex-hyperedge incidence matrix of $G$ then it is
straightforward that $A=\begin{bmatrix}
0 & M_{G} \\
M_{G}^t & 0 \\
\end{bmatrix}%
$ is the adjacency matrix of $B_{G}$.
\begin{proposition}
Let $G$ and $H$ be two fractionally isomorphic hypergraphs
($G\cong_f H$) and let $B_{G}$ and $B_{H}$ be the two bipartite
graphs that correspond to $G$ and $H$, respectively. Then
$B_{G}\cong_f B_{H}$.
\end{proposition}

\begin{proof}
Since $G$ and $H$ are fractionally isomorphic, then there exist
doubly stochastic matrices $S_{1}$ and $S_{2}$ such that
$S_{1}M_{G}=M_{H}S_{2}^{t}$ and $M_{G}S_{2}=S_{1}^{t}M_{H}$. We
define $A=\begin{bmatrix}
0 & M_{G} \\
M_{G}^t & 0 \\
\end{bmatrix}%
$, $B=\begin{bmatrix}
0 & M_{H} \\
M_{H}^t & 0 \\
\end{bmatrix}%
$ and  $S=\begin{bmatrix}
S_{1} &0 \\
0&S_{2}^t  \\
\end{bmatrix}$.
The matrix $S$ is doubly stochastic, $SA=%
\begin{bmatrix}
S_{1} & 0 \\
0& S_{2}^t  \\
\end{bmatrix}%
\begin{bmatrix}
0 &M_{G} \\
M_{G}^t & 0 \\
\end{bmatrix}%
=
\begin{bmatrix}
0 &S_{1}M_{G} \\
S_{2}^tM_{G}^t & 0 \\
\end{bmatrix}%
$,
and $BS=%
\begin{bmatrix}
0 &M_{H} \\
M_{H}^t & 0 \\
\end{bmatrix}%
\begin{bmatrix}
S_{1} & 0 \\
0& S_{2}^t  \\
\end{bmatrix}%
=
\begin{bmatrix}
0 &M_{H}S_{2}^t \\
M_{H}^tS_{1} & 0 \\
\end{bmatrix}%
$. Then $SA=BS$ and $B_{G}\cong_f B_{H}$.
\end{proof}

On the other hand, the converse is not true:  $B_{G}\cong_f
B_{H}$, does not imply that $G$ and $H$ are fractionally
isomorphic. Indeed, for every hypergraph $H$, $B_{H} = B_{H^*}$
(as graphs, even if the sets have different meaning in the
representation) and, in general, $H$ and $H^*$ have different
number of vertices and edges, so they are not fractionally
isomorphic.

We can find also counterexamples with the same number of vertices
and hyperedges. Let $H_1$ be the complete graph on four vertices
and $H_2$ be the $5$-vertex graph consisting on an induced path of
four vertices plus a universal vertex (known as \emph{gem}). Both
the disjoint unions $H_1 \cup H_2^*$ and $H_1^* \cup H_2$ have
$11$ vertices and $11$ hyperedges. It is clear that $B_{H_1 \cup
H_2^*} = B_{H_1^* \cup H_2}$ (as graphs). However, their degree
sequences differ: $d_1(H_1 \cup H_2^*) =
\{2,2,2,2,2,2,2,3,3,3,3\}$ while $d_1(H_1^* \cup H_2) =
\{2,2,2,2,2,2,2,2,3,3,4\}$, so they are not fractionally
isomorphic.

\section{Fractional invariants}

In this section we will show that several fractional invariants of
graphs and hypergraphs are preserved by fractional (hyper)graph
isomorphism. We will deal with fractional versions of packing
(independent set), edge covering, matching, and transversal
(vertex covering), which have been widely studied in the
literature (see, for example,
\cite{A-K-Z-fracMatch,Ber-hypcov,B-Pul-fracmatch,C-K-O-fracmatch,C-F-G-G-fracov,Cor-comb-opt,C-H-fracmatch,Dom-thesis,Far-bal,Fur-framhyp,Fur-match,L-L-fracmatch,Lov-Plu-matching,L-M-D-F-frac,M-S-T-fracmatch,Pul-fracmatch,Schrijver03,Sch-Sey-Lov,U-S-fgt}).

Given a hypergraph $H = (S,X)$, a \emph{covering} of $H$ is a
collection of hyperedges $X_1, X_2, \dots, X_j$ so that $S =  X_1
\cup \dots \cup X_j$. The least $j$ for which this is possible,
the smallest size of a covering, is called the \emph{covering
number} of $H$ and is denoted $k(H)$. An element $s \in S$ is
called \emph{exposed} if it is in no hyperedge. If $H$ has an
exposed vertex, then no covering of $H$ exists and $k(H) =
\infty$. The covering problem can be formulated as an integer
program, ``minimize $\textbf{1}^t \cdot x$ subject to $M_H \cdot x
\geq \textbf{1}$, $x \in \{0,1\}^{|X|}$''. Furthermore, the
variables in this and subsequent linear programs are tacitly
assumed to be nonnegative. This integer problem can be relaxed to
calculate $k_{f}(H)$ (the \emph{fractional covering number} of
$H$), as the linear program ``minimize $\textbf{1}^{t}\cdot x$
subject to $M_{H}\cdot x\geq \textbf{1}$, $x \geq 0$'' (each $x_i$
can take any real nonnegative value). Any feasible solution of the
integer program is also a feasible solution of the linear program,
so $k(H)\geq k_{f}(H)$.

A \emph{packing} of an hypergraph $H=(S,X)$ is a subset of
vertices $Y\subseteq S$ with the property that no two elements of
$Y$ are in the same hyperedge of $H$. The \emph{packing number}
$p(H)$ is the maximum number of elements that a packing can have,
an can be formulated by the integer program ``maximize
$\textbf{1}^{t}y$ subject to $M_{H}^{t}y\leq \textbf{1}$, $y \in
\{0,1\}^{|S|}$''. Again, relaxing this problem, we can compute
$p_{f}(H)$ (the \emph{fractional packing number} of $H$) as the
following linear program: ``maximize $\textbf{1}^{t}y$ subject to
$M_{H}^{t}y\leq \textbf{1}, y\geq 0$''. In the same way, we have
that $p_{f}(H)\geq p(H)$. Notice also that if $H$ has an exposed
vertex $i$, then $p_{f}(H)=\infty$, since the value of $x_i$ is
not upper bounded.

The linear programs above are dual, so $p_{f}(G)=k_{f}(G)$ (see,
for example, \cite{Ch-LP}). Then, we have $p(G)\leq
p_{f}(G)=k_{f}(G)\leq k(G)$.

In a  graph $G$ without isolated vertices (seeing it as a
2-uniform hypergraph), the packing number $p(G)$ corresponds to
the independence number $\alpha (G)$. So, we can define
$\alpha_{f}(G)=p_{f}(G)$. In the same way, $k(G)$ is the minimum
number of edges we can choose to cover every vertex in $G$, the
\emph{edge cover number} of $G$, and $k_f(G)$ is a fractional
version of it.

Given a graph $G$, we can construct a hypergraph $H_{G}$ where the
hyperedges of $H_{G}$ are the independent sets of $G$. Given
$A_G$, the adjacency matrix of $G$, we can compute $M_{H_{G}}$. In
this way, the chromatic number of $G$ corresponds to the covering
number of $H_{G}$ ($\chi (G)=$ $k(H_{G})$). In the same way, we
can define $\chi_{f}(G)=k_{f}(H_{G})$ and it is possible to show
that this definition coincides with other ways of define the
\emph{fractional chromatic number} of $G$~\cite{U-S-fgt}. The
clique number of $G$ corresponds to the packing number of $H_{G}$
($\omega (G)=$ $p(H_{G})$), and the \emph{fractional clique
number} $\omega_f(G)$ is defined as $p_f(H_{G})$. In particular,
$\omega(G) \leq \omega_f(G) = \chi_f(G) \leq \chi(G)$.

We can also consider the dual hypergraph $H^*$ of the hypergraph
$H$. A \emph{matching} of $H$ is a set of pairwise disjoint
hyperedges, and the \emph{matching number} of $H$, denoted
$\mu(H)$, is the maximum pairwise disjoint number of hyperedges of
$H$. So, $\mu(H)$ is the packing number of $H^*$. Also, we can
define $\mu_{f}(H)=p_{f}(H^*)$, and this definition coincides with
other ways to define the fractional matching number of graphs and
hypergraphs~\cite{Be-hyper,U-S-fgt}. A graph $G$ admits a
\emph{perfect fractional matching} when
$\mu_{f}(G)=\frac{1}{2}|V(G)|$.

A \emph{transversal} of $H$ is a set of vertices such that each
hyperedge contains at least one of them, and the \emph{transversal
number} of $H$, denoted $\tau(H)$, is the minimum cardinality of a
transversal of $H$. So, $\tau(H)= k(H^*)$. Also, we can define
$\tau_{f}(H)=k_{f}(H^*)$. For graphs, the transversal (integer or
fractional) is better known as \emph{vertex cover} of the graph.

The main result in this section is the following.

\begin{theorem}\label{thm:cover} Let $G$ and $H$ be hypergraphs.
If $G\cong_f H$, then $k_{f}(G)=k_{f}(H)$.
\end{theorem}

\begin{proof}
Since $G\cong_f H$, there exist doubly stochastic matrices $S_{1}$
and $S_{2}$ such that $S_{1}M_{G}=M_{H}S_{2}^{t}$ and
$M_{G}S_{2}=S_{1}^{t}M_{H}$. We write the following linear
programs:

P1: ``Minimize $a=\textbf{1}^{t}\cdot z$ subject to $M_{H}z\geq
\textbf{1}$ with $z\geq
0$''. \\

P2: ``Minimize $b=\textbf{1}^{t}\cdot w$ subject to
$S_{1}^{t}M_{H}w\geq \textbf{1}$ with
$w\geq 0$''. \\

P3: ``Minimize $c=\textbf{1}^{t}\cdot u$ subject to
$M_{G}S_{2}u\geq \textbf{1}$ with $u\geq
0$''. \\

P4: ``Minimize $d=\textbf{1}^{t}\cdot v$ subject to $M_{G}v\geq
\textbf{1}$ with $v\geq
0$''. \\

Let $a^*$, $b^*$, $c^*$, and $d^*$ be the optimal values of each
of the problems. It is straightforward that P2 and P3 are the same
problem using that $S_{1}^{t}M_{H}=M_{G}S_{2}$, thus $b^*=c^*$. On
the other hand, $a^*\geq b^*$ because every feasible solution of
P1 is a feasible solution of P2: if $z$ satisfies $M_{H}z\geq
\textbf{1}$, then $S_{1}^{t}(M_{H}z)\geq S_{1}^{t}\cdot
\textbf{1}=\textbf{1}$ ($S_{1}$ is a doubly stochastic matrix, and
in particular its entries are nonnegative). In the same way,
$c^*\geq d^*$, because for every solution $u$ of P3 we have a
feasible solution $v$ of P4 with the same objective function: if
$u$ satisfies $M_{G}S_{2}u\geq \textbf{1}$, then defining
$v=S_{2}u$, we obtain $M_{G}v\geq \textbf{1}$ and
$\textbf{1}^{t}\cdot v=\textbf{1}^{t}\cdot
S_{2}u=\textbf{1}^{t}\cdot u$. So, we have $a^*\geq b^*=c^*\geq
d^*$ and $a^*=k_{f}(H)$ and $d^*=k_{f}(G)$, then $k_{f}(H)\geq
k_{f}(G)$.

Using the same idea, we will write two additional linear programs: \\

P5: ``Minimize $e=\textbf{1}^{t}\cdot x$ subject to
$S_{1}M_{G}x\geq \textbf{1}$ with $x\geq
0$''. \\

P6: ``Minimize $f=\textbf{1}^{t}\cdot y$ subject to
$M_{H}S_{2}^{t}y\geq \textbf{1}$ with
$y\geq 0$''. \\

Let $e^*$ and $f^*$ be their optimal values. It holds that
$d^*\geq e^*=f^* \geq a^*$, so $k_{f}(G)\geq k_{f}(H)$. Therefore,
$k_{f}(H)=k_{f}(G)$.
\end{proof}

As a consequence, we have the following.

\begin{corollary}
\label{thm:frac-match-iso} Let $G$ and $H$ be hypergraphs. If
$G\cong_f H$, then $p_{f}(G)=p_{f}(H)$, $\mu_{f}(G)=\mu_{f}(H)$,
and $\tau_{f}(G)=\tau_{f}(H)$.
\end{corollary}

\begin{proof}
By Theorem~\ref{thm:cover} and the equality of the parameters,
$p_{f}(G)=k_{f}(G)=k_{f}(H)=p_{f}(H)$. By
Proposition~\ref{basics}, $G^*\cong_f H^*$. So, the previous
equality holds also for the dual graphs, and
$p_{f}(G^*)=k_{f}(G^*)=k_{f}(H^*)=p_{f}(H^*)$. By definition,
$\mu_{f}(G)=\tau_{f}(G)=\tau_{f}(H)=\mu_{f}(H)$.
\end{proof}

In terms of graphs, we have the following.

\begin{corollary}
\label{thm:frac-match-iso-graphs} Let $G$ and $H$ be graphs. If
$G\cong_f H$, then they have the same fractional independence
number, fractional edge and vertex covering numbers, and
fractional matching number. In particular, $G$ has a perfect
fractional matching if and only if $H$ has a perfect fractional
matching.
\end{corollary}

A sufficient condition for a graph to admit a perfect fractional
matching is the following. If the coarsest equitable partition $P=
V_1, \dots, V_k$ of a graph $G$ is such that for every $1 \leq i
\leq k$, $G[V_i]$ is $r_i$-regular with $r_i > 0$, then $G$ has a
perfect fractional matching. Namely, we assign a value
$\frac{1}{r_i}$ to every internal edge of $V_i$, and $0$ to the
edges that joint different parts of $P$.

Concerning the fractional independence number, there are other
ways of defining a fractional version of the independence number
of a graph. For example, relaxing the clique formulation of
it~\cite{Chvatal75}. That is, the fractional packing number of the
hypergraph whose vertices are the vertices of the graph and whose
hyperedges are the cliques of the graph. We will denote this
fractional relaxation of the independence number of a graph $G$ by
$\alpha^c_f(G)$. The dual problem for this formulation is the
(fractional) \emph{clique cover} of a graph (denoted by
$\theta(G)$, resp. $\theta_f(G)$), instead of the (fractional)
edge cover.

Notice that for the disjoint union of two triangles, $2 =
\alpha(2C_3) \leq \alpha^c_f(2C_3) = \theta_f(2C_3) \leq
\theta(2C_3) = 2$, so $\alpha^c_f(2C_3) = \theta_f(2C_3) = 2$, and
for the cycle of length six, $3 = \alpha(C_6) \leq \alpha^c_f(C_6)
= \theta_f(C_6) \leq \theta(C_6) = 3$, so $\alpha^c_f(C_6) =
\theta_f(C_6) = 3$. Therefore, the fractional clique cover and the
fractional relaxation of the clique formulation of the
independence number are not invariant under fractional graph
isomorphism.

Also, for the disjoint union of two triangles, $3 = \omega(2C_3)
\leq \omega_f(2C_3) = \chi_f(2C_3) \leq \chi(2C_3) = 3$, so
$\omega_f(2C_3) = \chi_f(2C_3) = 3$, and for the cycle of length
six, $2 = \omega(C_6) \leq \omega_f(C_6) = \chi_f(C_6) \leq
\chi(C_6) = 2$, so $\omega_f(C_6) = \chi_f(C_6) = 2$. Therefore,
the fractional chromatic and clique numbers are not invariant
under fractional graph isomorphism.

The domination number and the total domination number of a graph
can be viewed also as the covering numbers of associated
hypergraphs. Given a graph $G$, let $H_{N(G)}$ be the hypergraph
whose vertex set is $V(G)$ and whose hyperedges are the open
neighborhoods of the vertices of $G$. Then it is not difficult to
see that $\Gamma(G) = k(H_{N(G)})$. Alternatively, if we consider
the hypergraph $H_{N[G]}$, whose hyperedges are the closed
neighborhoods of the vertices of $G$, then $\gamma(G) =
k(H_{N[G]})$. The fractional domination number and fractional
total domination number are defined, as in other cases, as
$\gamma_f(G) = k(H_{N[G]})$ and $\Gamma_f(G) = k_f(H_{N(G)})$, and
this coincides with other ways of defining fractional
domination~\cite{U-S-fgt}.

Domke proved the following result.

\begin{theorem}\cite{Dom-thesis}\label{thm:fdom-reg} If $G$ has $n$ vertices and is $k$-regular, then $\gamma_f(G)
= \frac{n}{k + 1}$ and $\Gamma_f(G) = \frac{n}{k}$.
\end{theorem}

In particular, two $n$-vertex $k$-regular graphs have the same
fractional domination and total domination numbers. We generalize
this to fractionally isomorphic graphs.

\begin{theorem}\label{thm:fdom-iso} Let $G$ and $H$ be fractionally isomorphic graphs. Then $\gamma_f(G)
= \gamma_f(H)$ and $\Gamma_f(G) = \Gamma_f(H)$.
\end{theorem}

\begin{proof}
Let $A$ and $B$ be square symmetric matrices such that there
exists a double stochastic matrix $S$ with $AS = SB$. Then,
defining $S_1 = S^t$, $S_2 = S$, it holds  $S_1A = S^t A = (A^t
S)^t = (AS)^t = (SB)^t = B^t S^t =  BS^t = BS_2^t$, and $AS_2 = AS
= SB = S_1^t B$. Notice that if $G$ (and thus $H$) has an isolated
vertex, then $\Gamma_f(G) = \Gamma_f(H) = \infty$. Otherwise, the
vertex-hyperedge incidence matrix of the hypergraph $H_{N(G)}$ of
$G$ is exactly its adjacency matrix $A_G$, and the same holds for
$H$. In particular, these are symmetric matrices. Since $G\cong_f
H$, then there exists a double stochastic matrix $S$ with $A_G S =
SA_H$. By the observation above, there exist double stochastic
matrices $S_1$ and $S_2$ such that $S_1 A_G = A_H S_2^t$, and $A_G
S_2 = S_1^t A_H$. Thus, $H_{N(G)}\cong_f H_{N(H)}$ and by
Theorem~\ref{thm:cover}, $\Gamma_f(G) = k_f(H_{N(G)}) =
k_f(H_{N(H)}) = \Gamma_f(H)$.

Similarly, independently of the existence of isolated vertices,
the vertex-hyperedge incidence matrix of the hypergraph $H_{N[G]}$
of $G$ is $A_G + I$, where $I$ is the identity matrix of the
appropriate dimension (and the same for $H$). These are also
symmetric matrices, and we have
\begin{eqnarray*}
A_G S & = & S A_H\\
A_G S + S & = & S A_H + S\\
(A_G+I) S & = & S (A_H+I)\\
\end{eqnarray*}
As for the previous case, this implies $H_{N[G]}\cong_f H_{N[H]}$
and by Theorem~\ref{thm:cover}, $\gamma_f(G) = k_f(H_{N[G]}) =
k_f(H_{N[H]}) = \gamma_f(H)$.
\end{proof}

\section{Conclusions}

The computational complexity of the recognition problem of graph
isomorphism is an open question. It leads to define relaxations of
the usual isomorphism
 that can be computed efficiently, in order to discard negative instances. One of
these relaxations is the fractional isomorphism. Two graphs $G$
and $H$ are fractionally isomorphic if there exists a double
stochastic matrix $S$ such that $AS=SB$, where $A$ and $B$ are the
adjacency matrices of $G$ and $H$, respectively. The recognition
problem of fractional isomorphism has polynomial complexity, as it
can be modeled by linear programming, which is
polynomial~\cite{Kha-lp}.

A question this work tried to answer is whether there exists a
relationship between the fractional graph isomorphism and some
fractional parameters, such as fractional packing and covering.

To compute these parameters, it is necessary to work with the
vertex-edge incidence matrix instead of the adjacency matrix. So,
it is natural to ask how to describe the fractional graph
isomorphism in terms of these matrices. We obtain that $G$ and $H$
are fractionally isomorphic if and only if there exist doubly
stochastic matrices $S_{1}$ and $S_{2}$ such that
$S_{1}M_{G}=M_{H}S_{2}^{t}$ and $M_{G}S_{2} =S_{1}^{t}M_{H}$, with
$M_{G}$ and $M_{H}$ the vertex-edge incidence matrices of $G$ and
$H$, respectively. This later definition can be extended to
hypergraphs, and with this definition, the recognition of the
fractional hypergraph isomorphism has polynomial complexity as
well. We showed that the fractional covering and packing number,
and the fractional matching and transversal number are invariants
of the fractional hypergraph isomorphism. In particular, for
graphs, the fractional edge cover and independence number, and the
fractional matching and vertex cover number are invariants of the
fractional isomorphism. Moreover, the fractional domination and
total domination are invariants of the fractional graph
isomorphism. This is not the case of fractional chromatic, clique,
and clique cover numbers. In this way, most of the classical
fractional parameters are classified with respect to their
invariance under fractional graph isomorphism.

\bigskip

\noindent \textbf{Acknowledgements:} This work was partially
supported by UBACyT Grant 20020130100808BA (Argentina). We thank
Gianpaolo Oriolo and Paolo Ventura for introducing us into this
topic, and Martin Milanic for an interesting question about the
relation with the fractional isomorphism of the bipartite graphs
associated to the hypergraphs.


\end{document}